\documentclass[11pt]{article}
\usepackage{amsthm}
\usepackage{amsmath}
\usepackage{amssymb}
\usepackage{latexsym}
\usepackage{amsfonts}
\usepackage{mathrsfs}
\usepackage{listings}
\newtheorem{theorem}{\bf Theorem}[section]

\newtheorem{lemma}[theorem]{\bf Lemma}
\usepackage[none]{hyphenat}
\begin{document}
\title{On the Exponential Diophantine Equation $\left(F_{m+1}^{(k)}\right)^x - \left(F_{m-1}^{(k)}\right)^x = F_{n}^{(k)}$}
\author{Hayat Bensella, Bijan Kumar Patel and Djilali Behloul}
\date{}
\maketitle
\begin{abstract} \noindent
In this paper, we explicitly find all solutions of the title Diophantine equation,  using lower bounds for linear forms in logarithms and properties of continued fractions. Further, we use a version of the Baker-Davenport reduction method in Diophantine approximation, due to Dujella and Peth$\ddot{\textrm{o}}$. This paper extends the previous work of \cite{Patel}.
\end{abstract}

\section{Introduction}
\noindent
Let $\{F_n\}_{n\geq 0}$ be the Fibonacci sequence given by 
\[
F_{n+2} = F_{n+1} + F_n~~ {\rm for} ~n\geq 0
\]
with initials $F_0 = 0$ and $F_1 = 1$. 

The Fibonacci numbers are celebrated for possessing wonderful and amazing properties \cite{Koshy}. Hundreds of Fibonacci properties have been developed over the centuries by numerous mathematicians and number enthusiasts. Among such fabulous properties, in 1876, E. Lucas \cite{Lucas} showed that for all $n \geq 1$,
\begin{align}
F_n^2+F_{n+1}^2&=F_{2n+1}, \label{naive} \\
F_{n+1}^{2} - F_{n-1}^{2} &= F_{2n}. \label{cons}
\end{align}

Diophantine equations related to sums of powers of two terms of a given linear recurrence sequence, were studied by several authors. For instance, motivated by the naive identity \eqref{naive},
which tells us that the sum of the square of two consecutive Fibonacci numbers is still a Fibonacci number, Marques and Togb\'{e} \cite{Marques} showed that, if $x \geq 1$ is an integer such that $F_{n}^x + F_{n+1}^x$ is a Fibonacci number for all sufficiently large $n$, then $x \in \{1, 2 \}.$ Later, Luca and Oyono \cite{Luca} solved this problem completely by showing that the  Diophantine equation
$
F_m^s+F_{m+1}^s=F_n
$
has no solutions $(m,n,s)$ with $m\geq 2$ and $s \geq 3$. Subsequently, 
Chaves and Marques \cite{Chaves} proved the Diophantine equation
\[
\left(F_{m}^{(k)}\right)^s + \left(F_{m+1}^{(k)}\right)^s = F_{n}^{(k)}
\] 
in $k$-generalized Fibonacci numbers, showing that it has no positive integer solution $(n, m, k)$ with $k \geq 3$ and $n \geq 1$. Then Ruiz and Luca \cite{Ruiz} solved this Diophantine equation completely by showing that it has no solutions.

Patel and Chaves \cite{Patel} studied analogous of \eqref{cons} in higher powers and solved the problem completely by showing that the Diophantine equation
\begin{equation}\label{eq001}
F_{n+1}^{x} - F_{n-1}^{x} = F_{m},
\end{equation} 
has only non-negative integer solutions $(m, n, x) = (2n, n, 2), (1, 1, x), (2, 1, x), (0, n, 0)$. Consequently, G$\acute{\text{o}}$mez et al. \cite{Gomez} considered an extra exponent $y$ in \eqref{eq001} and investigated the equation 
\[
F_{n+1}^{x} - F_{n-1}^{x} = F_{m}^{y}
\]
in positive integers $(n,m,x,y)$. Patel and Teh \cite{Patel1} found all the solutions of the exponential Diophantine equation $F_{n+1}^x + F_{n}^x - F_{n-1}^x = F_m$ in non-negative integers $(m, n, x)$, which tells us that the sum of the power of three consecutive Fibonacci numbers is still a Fibonacci number.

Let $k \geq 2$ be an integer. One of numerous generalizations of the Fibonacci sequence, which is called the $k$-generalized Fibonacci sequence $\{F_{n}^{(k)} \}_{n \geq -(k-2)}$ is given by the recurrence

\begin{equation}\label{eq1}
  F_{n}^{(k)} = F_{n-1}^{(k)} + F_{n-2}^{(k)} + \dots + F_{n-k}^{(k)} = \sum_{i=1}^{k} F_{n-i}^{(k)} ~~\text{for all}~n \geq 2,  
\end{equation}
with initial conditions $F_{-(k-2)}^{(k)} = F_{-(k-3)}^{(k)} = \dots = F_{0}^{(k)} = 0$ and $F_{1}^{(k)} = 1$. Here, $F_{n}^{(k)}$ denotes the $n$th $k$-generalized Fibonacci number.

Note that for $k = 2$, we have $F_{n}^{(2)} = F_n$, the $n$th Fibonacci number. 
For $k = 3$, we have $F_{n}^{(3)} = T_n$, the $n$th Tribonacci numbers. They are followed by the Tetranacci numbers for $k = 4$, and so on.

The first observation is that the first $k + 1$ non-zero terms in $F_{n}^{(k)}$ are powers of $2$, namely
\[
F_{1}^{(k)}=1, F_{2}^{(k)}=1, F_{3}^{(k)}=2, F_{4}^{(k)}=4, \dots, F_{k+1}^{(k)} = 2^{k-1},
\]
while the next term in the above sequence is $F_{k+2}^{(k)} = 2^k -1$. Thus, we have that
\[
F_{n}^{(k)} = 2^{n-2}~\text{holds for all}~2 \leq n \leq k+1.
\]
Indeed, observe that recursion \eqref{eq1} implies the three-term recursion
\[
F_{n}^{(k)} = 2F_{n-1}^{(k)} - F_{n-k-1}^{(k)}~\text{for all}~n \geq 3,
\]
which shows that the $k$-Fibonacci sequence grows at a rate less than $2^{n-2}$.  In fact, the inequality $F_{n}^{(k)} < 2^{n-2}$ holds for all $n \geq k + 2$ (see \cite{Bravo}, Lemma 2).

In this paper, we study an analogue of \eqref{eq001} when the Fibonacci sequence is replaced by $k$-generalized Fibonacci sequence. More precisely, our main result is the following.
\begin{theorem}\label{thm1}
The Diophantine equation 
\begin{equation}\label{eq00}
\left(F_{m+1}^{(k)} \right)^x - \left(F_{m-1}^{(k)} \right)^x = F_{n}^{(k)}    
\end{equation}
has no positive integer solutions $m, n, k$ and $x \geq 2$ with $3 \leq k \leq \min \{m, \log x \}$.
\end{theorem}

Before getting into details, we give a brief description of our method. We first use Matveev's result \cite{Matveev} on linear forms in logarithms to obtain an upper bound for $x$ in terms of $m$. When $m$ is small, we use Dujella and Peth\"{o}'s result \cite{Dujella} to decrease the range of possible values that allow us to treat our problem computationally. When $m$ is large, we use a lower bound for linear forms in two logarithms \cite{Laurent} to get an absolute upper bound for $x$. In the final step, we use continued fractions to lower the bounds and then complete the calculations. 

\section{Auxiliary results}
\noindent
We recall some of the facts and properties of the $k$-generalized Fibonacci sequence which will be used after. Note that the characteristic polynomial of the $k$-generalized Fibonacci sequence is
\[
\Psi_{k}(x) = x^k - x^{k-1} - \dots - x - 1.
\]
$\Psi_{k}(x)$ is irreducible over $\mathbb{Q}[x]$ and has just one root outside the unit circle. It is real and positive, so it satisfies $\alpha(k) > 1$. The other roots are strictly inside the unit circle. Throughout this paper, $\alpha := \alpha(k)$ denotes that single root, which is located between $2(1-2^{-k})$ and $2$ (see \cite{Miyazaki}). To simplify notation, we will omit the dependence on $k$ of $\alpha$.  

Dresden \cite{Dresden} gave a simplified Binet-like formula
for $F_{n}^{(k)}$:
\begin{equation}\label{eq0}
    F_{n}^{(k)} = \sum_{i=1}^{k} \frac{\alpha_{i}-1}{2+(k+1)(\alpha_{i}-2)} \alpha_{i}^{n-1},
\end{equation}
where $\alpha = \alpha_{1}, \dots, \alpha_{k}$ are the roots of $\Psi_{k}(x)$. He also showed that the contribution of the roots which are inside the unit circle to the right-hand side of \eqref{eq0} is very small. More precisely, he proved that
\[
\Big| F_{n}^{(k)} - \frac{\alpha -1 }{2+(k+1)(\alpha - 2)} \alpha^{n-1} \Big| < \frac{1}{2}~~{\rm for~ all}~~n \geq 1.
\]
The following inequality is proved by Bravo and Luca \cite[Lemma~1]{Bravo}.
\begin{lemma}\label{lm1}
The inequality 
\[
\alpha^{n-2} \leq F_{n}^{(k)} \leq \alpha^{n-1}
\]
holds for all $n \geq 1$.
\end{lemma}

The following result is derived in \cite[Eq. 3.14]{Chaves}.
\begin{lemma}\label{lm4}
Let $y_{m} := \frac{|E_{m}(k)|s}{g \alpha^{m-1}}$ for $m \geq 1395$. Then
\[
|(F_{m}^{(k)})^s - g^s \alpha^{(m-1)s}| < 2 y_{m} g^s \alpha^{(m-1)s}.
\]
\end{lemma}

The following result is proved by S$\acute{\text{a}}$nchez and Luca \cite[Lemma 7]{Sanchez}. 
\begin{lemma} \label{lm5} 
If $ r \ge 1 , T> (4r^{2})^{r}$, and  $T> a/(\log a)^{r}$. Then  
 \[ 
 a <2^{r}T(\log T)^{r}.
 \]
\end{lemma}

The following lemma is useful in the subsequent result.
\begin{lemma}\label{lm2}
For all $m \geq 3$ and $k \geq 3$, we have
\[
F_{m-1}^{(k)}/F_{m+1}^{(k)} \leq 3/7.
\]
\end{lemma}
\begin{proof}
Indeed, we have
$7 F_{m-1}^{(k)} \leq 3 F_{m+1}^{(k)} = 3 \left( F_{m}^{(k)} + F_{n-1}^{(k)} + \dots + F_{m-(k-1)}^{(k)} \right)$. Further simplification, we obtain 
\[
4 F_{m-1}^{(k)} = 4 \left(F_{m-2}^{(k)} + F_{m-3}^{(k)} + \dots + F_{m-(k+1)}^{(k)} \right) \leq 3 F_{m}^{(k)} + 3 \left(F_{m-2}^{(k)} + F_{m-3}^{(k)} + \dots + F_{m-(k-1)}^{(k)} \right)
\] 
and then $F_{m-1}^{(k)} + 3 F_{m-k}^{(k)} + 3 F_{m-(k+1)}^{(k)} \leq 3 F_{m}^{(k)}.$ Furthermore,
\[
0 \leq 2F_{m-1}^{(k)} + 3 \left(F_{m-2}^{(k)} + F_{m-3}^{(k)} + \dots + F_{m-(k+2)}^{(k)} \right)
\]
for $m \geq 3$, which follows the result.
\end{proof}

In order to prove our main result, we use a few times a Baker-type lower bound for a non-zero linear forms in logarithms of algebraic numbers. We state a result of Matveev \cite{Matveev} about the general lower bound for linear forms in logarithms, but first, recall some basic notations from algebraic number theory.

Let $\eta$ be an algebraic number of degree $d$ with minimal primitive polynomial 
\[
f(X):= a_0 X^d+a_1 X^{d-1}+ \cdots +a_d = a_0 \prod_{i=1}^{d}(X- \eta^{(i)}) \in \mathbb{Z}[X],
\]
where the $a_i$'s are relatively prime integers, $a_0 >0$, and the $\eta^{(i)}$'s are conjugates of $\eta$. Then
\begin{equation}\label{eq03}
h(\eta)=\frac{1}{d}\left(\log a_0+\sum_{i=1}^{d}\log\left(\max\{|\eta^{(i)}|,1\}\right)\right)
\end{equation}
is called the \emph{logarithmic height} of $\eta$.

With the established notations, Matveev (see  \cite{Matveev} or  \cite[Theorem~9.4]{Bugeaud1}), proved the ensuing result.

\begin{theorem}\label{th2}
Assume that $\gamma_1, \ldots, \gamma_t$ are positive real algebraic numbers in a real algebraic number field $\mathbb{K}$ of degree $D$, $b_1, \ldots, b_t$ are rational integers, and 
\[
\Lambda :=\gamma_1^{b_1}\cdots\gamma_t^{b_t}-1,
\]
is not zero. Then
\[
|\Lambda| \geq \exp\left(-1.4\cdot 30^{t+3}\cdot t^{4.5}\cdot D^2(1+\log D)(1+\log B)A_1\cdots A_t\right),
\]
where
\[
B\geq \max\{|b_1|,\ldots,|b_t|\},
\]
and
\[
A_i\geq \max\{Dh(\gamma_i),|\log \gamma_i|, 0.16\}, ~ \text{for all} ~ i=1,\ldots,t.
\]
\end{theorem}

 When $t = 2$ and $\gamma_{1}, \gamma_{2}$ are positive and multiplicatively independent, we use the following result of
 Laurent, Mignotte and Nesterenko \cite{Laurent}. 
 
 Let $B _{1}, B_{2}$ be real numbers larger than $1$ such that
 \[
  \log B_{i} \ge \max \left\{h(\gamma_{i}),\frac{| \log \gamma_{i}|}{D}, \frac{1}{D} \right\}~ \text{ for }~ i= 1,2.
\]
Let $b_1, b_2$ be integers not both zero. 
Put
\[
\Gamma := b_{1} \log \gamma_{1} + b_{2} \log \gamma_{2}
\]
and 
\[
b^{\prime} := \frac{|b_{1}|}{D\log B_{2}} + \frac{|b_{2}|}{D\log B_{1}}.
\] 
\begin{theorem}[\cite{Laurent}, Corollary 2, pp. 288] \label{thm6}
With the above notations and assumptions, we have 
\[
\log |\Gamma| > -24.34 D^{4} \left( \max \left\{ \log b^{\prime}+0.14, \frac{21}{D},\frac{1}{2} \right\} \right)^{2}  \log B_{1} \log B_{2}.
\]
\end{theorem}
Note that $\Gamma \neq 0$ since $\gamma_{1}$ and $\gamma_{2}$ are multiplicatively independent and $b_1$ and $b_2$ are integers not both zero.

The following criterion of Legendre, a well-known result from the theory of Diophantine
approximation, is used to reduce the upper bounds on variables which are too large.
\begin{lemma}\label{Legendre}
Let $\tau$ be an irrational number, $\frac{p_0}{q_0}, \frac{p_1}{q_1}, \frac{p_2}{q_2}, \dots$ be all the convergents of the continued fraction of $\tau$, and $M$ be a positive integer. Let $N$ be a non-negative integer such that $q_N > M$. Then putting $a(M) := \max \{a_i: i=0,1,2,\dots, N \}$, the inequality
\[
\Bigm| \tau - \frac{r}{s} \Bigm| > \frac{1}{(a(M)+2)s^2},
\]
holds for all pairs $(r, s)$ of positive integers with $0 < s < M$.
\end{lemma}

Another result which will play an important role in our proof is due to Dujella and Peth\"{o} \cite[Lemma~5 (a)]{Dujella}. 

\begin{lemma} \label{lm3}
Let $M$ be a positive integer, let $p/q$ be a convergent of the continued fraction of the irrational $\gamma$ such that $q > 6M$, and let $A,B,\mu$ be some real numbers with $A>0$ and $B>1$. Let $\epsilon:=||\mu q||-M||\gamma q||$, where $||\cdot||$ denotes the distance from the nearest integer. If $\epsilon >0$, then there exists no solution to the inequality
\[
0< |u \gamma-v+\mu| <AB^{-u},
\]
in positive integers $u$ and $v$ with
\[
u \leq M \quad\text{and}\quad u \geq \frac{\log(Aq/\epsilon)}{\log B}.
\]
\end{lemma}

\section{Proof of Theorem \ref{thm1}}
By Lemma \ref{lm1}, we obtain
\begin{align*}
    \alpha^{n-1} \geq F_{n}^{(k)} = (F_{m+1}^{(k)})^x - (F_{m-1}^{(k)})^x & \geq \alpha^{(m-1)x} - \alpha^{(m-3)x} \\
    &= \alpha^{(m-2)x} (\alpha^{x} - \alpha^{-x}) > \alpha^{(m-2)x+1}
\end{align*}
and
\[
\alpha^{n-2} \leq F_{n}^{(k)} = (F_{m+1}^{(k)})^x - (F_{m-1}^{(k)})^x \leq \alpha^{mx} - \alpha^{(m-2)x} < \alpha^{mx},
\]
where we used that $\alpha^{x} -\alpha^{-x} > \alpha$ for $x \geq 2$. Thus,
\begin{equation}\label{eq01}
    (m-2)x+2 < n < mx+2.
\end{equation}
From \eqref{eq0}, we can write 
\begin{equation}\label{eq2}
    F_{n}^{(k)} = g \alpha^{n-1} + E_{n}(k),~~~~\text{where}~~|E_{n}(k)| < 1/2.
\end{equation}
Hence, \eqref{eq00} can be written as 
\begin{equation}\label{eq3}
    g\alpha^{n-1} - (F_{m+1}^{(k)})^x = - (F_{m-1}^{(k)})^x - E_{n}(k).
\end{equation}
Dividing \eqref{eq3} by $(F_{m+1}^{(k)})^x$ and taking absolute values and by Lemma \ref{lm2}, we obtain
\begin{equation}\label{eq4}
    \left|g \alpha^{n-1} (F_{m+1}^{(k)})^{-x} - 1 \right| < 2 \left(\frac{F_{m-1}^{(k)}}{F_{m+1}^{(k)}} \right)^x < \frac{2}{2.3^{x}}.
\end{equation}

In order to use Theorem \ref{th2}, we take $t:=3$, 
\[
\gamma_1 := F_{m+1}^{(k)}, \gamma_2 := \alpha, \gamma_3 := g,
\]
and 
\[
b_1 := -x, b_2 := n-1, b_3 := 1.
\]
Hence,
\[
\Lambda_{1} := g \alpha^{n-1} (F_{m+1}^{(k)})^{-x} - 1.
\]
To see $\Lambda_{1} \neq 0$, suppose $g \alpha^{n-1} (F_{m+1}^{(k)})^{-x} - 1 = 0,$ we would get a relation $(F_{m+1}^{(k)})^x = g \alpha^{n-1}$, which is contradict the identity \eqref{eq2}.

Note that $\mathbb{K} := \mathbb{Q}(\alpha)$ contains $\gamma_{1}, \gamma_{2}, \gamma_{3}$ and has $D = [\mathbb{K}:\mathbb{Q}] = k$. By the properties of the dominant root of $\Psi_{k}(x)$ and Lemma \ref{lm1} we have
$h(\gamma_{1}) = \log (F_{m+1}^{(k)}) \leq m \log \alpha < 0.7 m,$ and $h(\gamma_{2}) = (\log \alpha)/k < 0.7/k$. Bravo and Luca \cite{Bravo} showed that $h(\gamma_{3}) < 3 \log k$. Thus we can take $A_{1} := 0.7 km, A_{2} := 0.7$ and $A_{3} := 3k \log k$. From \eqref{eq01}, we can take $B:= n-1$.

From Matveev's theorem, we have a lower bound for $|\Lambda_1|$, which together with the upper bound given by \eqref{eq4} gives 
\begin{equation*}\label{eq5}
\exp\left(-1.4 \cdot 30^6 \cdot 3^{4.5} \cdot k^2 \cdot (1+\log k) \cdot (1+\log (n-1)) \cdot (0.7mk) \cdot 0.7 \cdot (3k \log k) \right) < \frac{2}{2.3^x}.
\end{equation*}
Hence, noting that $1 + \log (n-1) < 2 \log (mx+1)$, is true for all $n \geq 4$ and performing the calculations, we get 
\[
x < 1.04 \cdot 10^{12} m k^4 (\log k)^2 \log (mx+1),
\]
or equivalently 
\[
\frac{mx+1}{\log (mx+1)} < 1.1 \cdot 10^{12} m^2 k^4 (\log k)^2.
\]
Now using Lemma \ref{lm5} for $a:=mx+1$ and $T:= 1.1 \cdot 10^{12} m^2 k^4 (\log k)^2$, we have that
\begin{equation}\label{eq8}
x < 7.1 \cdot 10^{13} m^5 (\log m)^3,
\end{equation}
where we used the fact that $28 + 6 \log m + 2 \log \log m < 32 \log m$ for $m \geq 3$ and the presumption $k \leq m$. 
\subsection{The case for small $m \in [3, 1457]$}
In this case, we have
\[
x < 7.1 \cdot 10^{13} \cdot (1457)^5 \cdot \log (1457)^3 \leq  1.81 \cdot 10^{32}
\]
for $m \in [3, 1457]$. Thus, we obtain $n \leq mx+1 < 2.63\cdot 10^{35}$ and $k \leq \log x \leq 74$. Also note that $n \leq mx+1$, gives us $x \geq (n-1)/1457$.

We set
\[
\Gamma_{1} := (n-1) \log \alpha - \log \left( \frac{1}{g} \right) - x \log F_{m+1}^{(k)}.
\]
Thus, $\Lambda_1 = e^{\Gamma_1}-1$. Recall that, from \eqref{eq3}, we have $\Lambda_1 < 0$, which implies $\Gamma_1 < 0$. Now, since $|\Lambda_1| < 2/(2.3)^x \leq 0.379$ for $x \geq 2$, it follows that $e^{|\Gamma_1|} < 1.5$. Hence, we get
\[
0 < |\Gamma_1| < e^{|\Gamma_1|}-1 \leq e^{|\Gamma_1|}|e^{\Gamma_1}-1| <  \frac{3}{2.3^x}.
\]
Dividing the last inequality by $\log F_{m+1}^{(k)}$, and using that $x > (n-1)/1457$, we get 

\begin{equation} \label{dujpetho1}
0 < n\left(\frac{\log \alpha}{\log F_{m+1}^{(k)}}\right) - x - \left(\frac{\log (\alpha/g)}{\log F_{m+1}^{(k)}}\right) < 3.01 \cdot (2.3)^{\frac{-n}{1457}} ~.
\end{equation}

In order to use the reduction method, take
\[
\gamma_{m,k} := \frac{\log \alpha}{\log F_{m+1}^{(k)}},~~\mu_{m,k} := - \frac{\log (\alpha/g)}{\log F_{m+1}^{(k)}},~~A:= 3.01,~~B:= (2.3)^{\frac{1}{1457}}.
\]
The fact that $\alpha$ is a unit in $\mathcal{O}_{\mathbb{K}}$, the ring of integers of $\mathbb{K}$, ensures that $\gamma_{m,k}$ is an
irrational number. Let $q_{(t,m,k)}$ be the denominator of the $t$th convergent of the continued fraction of $\gamma_{m,k}$.

Taking $M := 2.64 \cdot 10^{35}$, we use {\it Mathematica} to get 
\[
\min_{\substack {3 \leq k \leq 74 \\ 3 \leq m \leq 1457}} q_{(700,m,k)} > 10^{319} > 6M \mbox{ \ and \ } \max_{\substack {3 \leq k \leq 74 \\ 3 \leq m \leq 1457}} q_{(700,m,k)} < 2.1 \cdot 10^{425}. \
\]
The maximal value of $M\lVert \gamma_{m,k} \cdot q_{(700,m,k)} \rVert < 10^{-284}$, whereas the minimal value of $\lVert \mu_{m,k} \cdot q_{(700,m,k)} \rVert> 5.29 \cdot 10^{-214}$. 
Also, for 
\[
\epsilon_{700,m,k}:= \lVert\mu_{m,k} \cdot q_{(700,m,k)}\rVert - 2.64 \cdot 10^{35} \lVert\gamma_{m,k} \cdot q_{(700,m,k)} \rVert,
\]
we obtain that 
\[
 \epsilon_{700,m,k} > 5.29 \cdot 10^{-214} , 
\]
which means that $\epsilon_{700,m,k}$ is always positive (this is not true for the denominator of $600$th convergent). Hence by Lemma \ref{lm3}, there are no integer solutions for \eqref{dujpetho1} when 
\[
\left \lfloor \frac{\log(3.01 \cdot 2.1 \cdot 10^{425}/5.29 \cdot 10^{-214})}{\log((2.3)^{\frac{1}{1457}})}\right \rfloor \leq n \leq 2.63 \cdot 10^{35}.
\]
It follows that $ 850196\leq n \leq 2.64 \cdot 10^{35}$  and therefore we have $n \leq 850195$. Consequently $x \leq 850192$, since $ x \leq (n-3)/(m-2)$. Also, using that $k \leq \log x$, we get $k \leq 13$. 

A computer search with {\it Mathematica} revealed that there are no solutions to \eqref{eq00} in the following range:
\[ 
3 \leq m\leq 1457,  3\leq k \leq 13 , 21\leq x \leq 850192~ \text{ and }~ (m-2)x+2 \leq n \leq mx+1.
\]
This completes the analysis of the case when $m$ is small.
\subsection{The case of  large $m$}
Set $y_{m} := \frac{|E_{m}(k)|x}{g \alpha^{m-1}}$. From \eqref{eq8} and $m \geq 1458$, we have
\begin{equation}\label{eq11}
    y_m < \frac{|E_{m}(k)| 7.1 \cdot 10^{13} m^5 (\log m)^3}{g \alpha^{m-1}} < \frac{1}{\alpha^{(m-1)/2}},
\end{equation}
where $7.1 \cdot 10^{13} m^5 (\log m)^3 < (7/4)^{\frac{m-1}{2}} < \alpha^{\frac{m-1}{2}}$ holds for $m \geq 1458$. In particular, $y_m < \alpha^{-728} < 10^{-30}$. Similarly,
\[
    y_{m-1} = \frac{|E_{m-1}(k)|x}{g \alpha^{m-2}} < \frac{1}{\alpha^{(m-1)/2}}~~\text{and}~~
    y_{m+1} = \frac{|E_{m+1}(k)|x}{g \alpha^{m}} < \frac{1}{\alpha^{(m-1)/2}}.
\]
Lemma \ref{lm4} is true if we replace $m$ by $m-1$ and $m+1$. Thus
\begin{equation}\label{eq12}
|(F_{m-1}^{(k)})^x - g^x \alpha^{(m-2)x}| < 2 y_{m-1} g^x \alpha^{(m-2)x}~~\text{and}~~|(F_{m+1}^{(k)})^x - g^x \alpha^{mx}| < 2 y_{m+1} g^x \alpha^{mx}.
\end{equation}
Now, we need to make a few algebraic manipulations in order to apply Theorem \ref{thm6}. Rewrite \eqref{eq00} as follows,
\[
g \alpha^{n-1} + E_{n}(k) = \left((F_{m+1}^{(k)})^x - g^x \alpha^{mx} \right) - \left((F_{m-1}^{(k)})^x - g^x \alpha^{(m-2)x} \right) + g^x \alpha^{mx} - g^x \alpha^{(m-2)x},
\]
which gives
\begin{align*}
|g \alpha^{n-1} - g^x \alpha^{mx} (1 - \alpha^{-2x})| &\leq |(F_{m+1}^{(k)})^x - g^x \alpha^{mx}| + |(F_{m-1}^{(k)})^x - g^x \alpha^{(m-2)x}| + |E_{n}(k)| \\
& < 2y_{m+1} g^x \alpha^{mx} + 2 y_{m-1} g^x \alpha^{(m-2)x} + \frac{1}{2}.
\end{align*}
Dividing $g^x \alpha^{mx}$ on both sides, we have
\begin{equation}\label{eq13}
    |g^{1-x} \alpha^{n-(mx+1)} - (1-\alpha^{-2x})| < 2y_{m+1} + 0.1 y_{m-1} + \frac{1}{2g^x\alpha^{mx}},
\end{equation}
where we used the fact $2 \alpha^{-2x} < 2 \left( \frac{7}{4} \right)^{-6} < 0.1$. Since
\[
    2g^x \alpha^{mx-\frac{m-2}{2}} > 2 \left(\frac{1}{2} \right)^x \left(\frac{7}{4} \right)^{mx - \frac{m-2}{2}} > 10^{531} > 10^3,
\]
we have $\left(2 g^x \alpha^{mx} \right)^{-1} < 0.001/ \alpha^{(m-2)/2}$. Using the last inequalities with \eqref{eq13}, we obtain
\begin{align}\label{upper}
|g^{1-x} \alpha^{n-(mx+1)} - (1-\alpha^{-x})| 
&< \frac{2}{\alpha^{\frac{m}{2}}} + \frac{0.1}{\alpha^{\frac{m-2}{2}}} + \frac{0.001}{\alpha^{\frac{m-2}{2}}} \nonumber \\
&< \frac{2.11}{\alpha^{\frac{m-2}{2}}}.
\end{align}
Hence, we conclude that
\begin{align}\label{eq14}
    |g^{1-x} \alpha^{n-(mx+1)} - 1| &< \frac{2.11}{\alpha^{\frac{m-2}{2}}} + \frac{1}{\alpha^{2x}} \nonumber \\ 
    &< \frac{3.11}{\alpha^l},
\end{align}
where $l:= \min \{ x, \frac{m-2}{2}\}$.

We apply Theorem \ref{thm6} with $t:=2$, $\lambda_{1}:=g$, $\lambda_{2}:= \alpha$ and $c_{1}:= 1-x$, $c_{2}:= n-(mx+1)$. The fact that $\lambda_{1} =g$ and $\lambda_{2} = \alpha$ are multiplicatively independent follows because $\alpha $ is a unit and $g$ is not. To see that this is not so, we perform a norm calculation of the element $g$ in $\mathbb{L} := \mathbb{Q}(\alpha)$. The norm of $g$ has been determined for all $k\ge 2$ in \cite{Fuchs} and the formula is
\begin{equation*}\label{eq15}
	|N_{\mathbb{L}/\mathbb{Q}} (g_{k}(\alpha))|=\frac{(k-1)^{2}}{2^{k+1}k^{k}-(k+1)^{k+1}}.
\end{equation*}
One can check that $|N_{\mathbb{L}/\mathbb{Q}} (g_{k}(\alpha))|< 1 $ for all $k \ge 2$ and therefore $g$ is not a unit for any $k$. So we can take $ \mathbb{K} := \mathbb{Q}(\alpha)$ which has degree $D = k$, $h(g)< 3 \log k, h(\alpha)=(\log \alpha)/k$. Moreover, we can take 
\[
\log B_{1} = 4\log k  > \max{\left\{h(g),\dfrac{|\log (g)|}{k},\dfrac{1}{k}\right\}}  
~~{\rm and}~~ 
\log B_{2} =  \max{\left\{h(\alpha),\dfrac{|\log (\alpha)|}{k},\dfrac{1}{k}\right\}} = \dfrac{1}{k}.
\]
Observe that $n\geq (m-2)x+3$, which implies $n-(mx+1)\geq -2(x-1)$ and $n \le mx+1$,  yields  $ n-(mx+1)\le 2(x-1)$. Hence $|n-(mx+1)|<2(x-1) < 2x$.
Thus   
\[
{b}^{\prime } = \frac{|(1-x)|}{k(\frac{1}{k})} + \frac{|n-(mx+1)|}{4k \log k}<x + \frac{2x}{4k \log k }<1.2x.
\]
By Theorem \ref{thm6}, we get 
\begin{equation*}\label{eq16}
|\Lambda_{2}|> \exp \left(    -24.34 \cdot k^{4} 
\left( \max {\left\{\log(1.2x)+0.14 , \frac{21}{k},\frac{1}{2}\right\}} \right)^{2} 4 \log k \left(\frac{1}{k}\right) \right).
\end{equation*}
Thus 
\begin{equation*}\label{eq17}
|\Lambda_{2}|> \exp \left(-97.4 \cdot k^{3}\log k  
\left( \max{\left\{\log(1.4x) , \frac{21}{k},\frac{1}{2}\right\}} \right)^{2}  \right),
\end{equation*}
where we used the fact that $\log \left(1.2 x\right)+0.14 =\log \left(1.2\exp(0.14) x \right)< \log (1.4x) $.
The last inequality together with \eqref{eq14} we get   
\begin{equation*}\label{eq18}
\frac{3.11}{\alpha^{l}} > \exp \left(-97.4 \cdot k^{3}\log k  
\left( \max{\left\{\log(1.4x) , \frac{21}{k},\frac{1}{2}\right\}} \right)^{2}  \right)
\end{equation*}
yielding
\begin{equation}\label{eq19}
\ l \log \alpha - \log 3.11 < 97.4 \cdot k^{3} \log k  
\left( \max {\left\{\log(1.4x), \frac{21}{k},\frac{1}{2}\right\}} \right)^{2}.
\end{equation}
Since $\log(1.4x) > 1.02 $, the maximum in the right-hand side cannot be $\frac{1}{2}$. Let the maximum in the right-hand side of \eqref{eq19}  be $21/k$, then we get 
\begin{equation*}\label{eq20}
	l < 174.1 \cdot k^{3}\log k \left(\frac{21}{k}\right)^2
\end{equation*}
yielding
\begin{equation*}\label{eq21}
	l < 7.7 \cdot 10^{4} k \log k.
\end{equation*}
If $l= x$, then 
\begin{equation}\label{eq22}
	x < 7.7 \cdot 10^{4} k \log k.
\end{equation}
If $l = \frac{m-2}{2}$, then 
\begin{equation*}\label{eq23}
	\frac{m-2}{2} < 7.7 \cdot 10^{4} k \log k
\end{equation*}
thus,
\begin{equation}\label{eq24}
	m<1.55 \cdot 10^{5} k \log k,
\end{equation}
where we used that $m-2 > \frac{m}{1.002}$ for $m \ge 1458$.

Now, assume that the maximum in the right-hand side of \eqref{eq19}  is $\log(1.4x)$
then,
\begin{equation*}\label{eq25}
	l<174.1 \cdot k^{3}\log k \left(\log (1.4x)\right)^2\\
\end{equation*}
yielding
\begin{equation*}\label{eq26}
	l<3.5 \cdot 10^{2} k^{3}\log k \left(\log x\right)^2.
\end{equation*}  
For the above inequality, we used that $\log (1.4x)<1.4\log x$ holds for $x\ge 2$. If $l=x$, then 
\begin{equation*}\label{eq28}
	x<3.5 \cdot 10^{2} k^{3}\log k (\log x )^{2}
\end{equation*}
thus,
\begin{equation*}\label{eq29}
\frac{x}{(\log x)^{2}}<3.5 \cdot 10^{2} k^{3}\log k.
\end{equation*}
Now we apply Lemma \ref{lm5} with $r:=2$ and $T := 3.5 \cdot 10^{2} k^{3}\log k$  to get 
\begin{equation*}\label{eq30}
	x<1.4 \cdot 10^{3} k^{3}\log k (3\log k +\log \log k +\log (3.5 \cdot 10^{2}))^2 
\end{equation*}
thus,
\begin{equation}\label{eq31}
x < 1.01 \cdot 10^{5} k^{3} (\log k)^3,
\end{equation}
where we used that 
$	
\left(3\log k + \log \log k + \log (3.5 \cdot 10^{2}) \right)^{2} < 72 (\log k)^{2} .
$

If $l= \frac{m-2}{2}$, then we use \eqref{eq8} to get 
\begin{align*} \label{eq32}
\frac{m-2}{2} <& 3.5 \cdot 10^{2} k^{3} \log k (\log x)^{2} \\ 
<& 3.5 \cdot 10^{2} k^{3} \log k (\log (7.1 \cdot 10^{13} m^{5} (\log m)^{3} ))^{2} \\ 
<& 3.5 \cdot 10^{2} k^{3}\log k (\log(7.1 \cdot 10^{13})+5\log m +3\log \log m)^{2} 
\end{align*}
thus, 
\begin{equation*}\label{eq33}
	\frac{m}{(\log m)^{2}}< 7.72 \cdot 10^{4} k^{3} \log k 
\end{equation*}
where we used that $ \left(\log (7.1 \cdot 10^{13})+ 5 \log m + 3 \log \log m \right) ^{2} < 1.1 \cdot 10^{2}(\log m)^{2}$  and     $m-2>m/1.002$  for  $m \ge 1458$ .
  
Using Lemma \ref{lm5} again we get an upper bound for $m$ in terms of $k$  
\begin{equation}\label{eq34}
	m< 5.6 \cdot 10^{7} k^{3} (\log k)^{3},
\end{equation}
where we used that $ \left(\log (7.72 \cdot 10^{4})+ 3 \log k +  \log \log k \right) ^{2} < 1.8 \cdot 10^{2}(\log k)^{2}$  and $m-2>m/1.002$  for  $m \ge 1458$ .

Comparing inequalities \eqref{eq22} with \eqref{eq31} and \eqref{eq24} with \eqref{eq34}, respectively, we
conclude that \eqref{eq31} and \eqref{eq34} always hold.\\
Now by \eqref{eq34} combined with \eqref{eq8}, we get 
\begin{equation}\label{eq35}
	x<1.85 \cdot 10^{56} k^{15}(\log k )^{18}         
\end{equation}
where we used that

$ \left(\log (5.6 \cdot 10^{7})+ 3 \log k + 3 \log \log k \right) ^{3} < 4.7 \cdot 10^{3}(\log k)^{3}$  and     $m-2>m/1.002$  for  $m \ge 1458$ .
Finally, comparing \eqref{eq31} with \eqref{eq35}, we conclude that \eqref{eq35} always hold and gives us an upper bound for $x$ in terms of $k$. Since $ k\le \log x $, we have 
\begin{equation*}\label{bound}
	x<1.85 \cdot 10^{56} (\log x)^{15}(\log \log x)^{18} 
\end{equation*}
which is true only for $x<2.27 \cdot 10^{105},~ \text{so  }k<\log (2.27 \cdot 10^{105})< 242.$ 

The previous bounds are too large, so we need to reduce them by using a criterion due to Legendre. First, we go to \eqref{eq14} and using that $ x \ge 20$ and $m\ge 1458 $, we get the following upper bound
\begin{equation}\label{eq36}
|\Lambda_{2}|<\frac{1}{\alpha^{2x}}+\frac{2.11}{\alpha^{\frac{m-2}{2}}}<\frac{1}{\alpha^{40}}+\frac{2.11}{\alpha^{728}}<1.9 \cdot 10^{-10}.
\end{equation}
Let
\[
\Gamma_{2}:= (x-1)\log (g^{-1})-(mx+1-n)\log\alpha.
\]
Equation \eqref{eq36} can be rewritten as $|e^{\Gamma_{2}-1}|<1.9 \cdot 10^{-10}$.  Since $|e^{\Gamma_{2}-1}|<1.9 \cdot 10^{-10}$, we get that $|e^{\Gamma_{2}}|<1.9 \cdot 10^{-10}+1$. Thus  
\begin{align*}
    |{\Gamma_{2}}|\le e^{|\Gamma_{2}|}|e^{\Gamma_{2}}-1|&<(1.9 \cdot 10^{-10}+1)| \Lambda_{2}| \\
    &< (1.9  \cdot 10^{-10}+1) \left( \frac{1}{\alpha^{2x}}+\frac{2.11}{\alpha^{(m-2)/2}} \right).
\end{align*}
By replacing $\Gamma_{2}$ in the above inequality by it's formula and dividing through by $(x-1)\log \alpha $, we obtain
\begin{equation}\label{eq37}
\left|\frac{\log (g^{-1})}{\log \alpha}-\frac{mx+1-n}{x-1}\right|<\frac{(1.9 \cdot 10^{-10}+1)}{(x-1) \log \alpha} \left( \frac{1}{\alpha^{2x}}+\frac{2.11}{\alpha^{(m-2)/2}}\right).
\end{equation}
Since $ m \ge 1458 $, we have $\alpha^{\frac{m-2}{2}}>(\frac{7}{4})^{728} > 2.32 \cdot 10^{71}x $.
Assume that $x>150$. Then $\alpha^{2x}>2.32 \cdot 10^{71}x$. Now \eqref{eq37} becomes
\begin{equation}\label{eq38}
\left| \frac{\log (g^{-1}) }{\log \alpha}-\frac{mx+1-n}{x-1}\right|<\frac{1}{4.14 \cdot 10^{70} (x-1)^{2}}.
\end{equation}
By the Lemma \ref{Legendre}, we infer that  $(mx+1-n)/(x-1)$ is a convergent of the continued fraction of $\beta_{k}=(\log (g^{-1}))/(\log \alpha) = [a_{0}^{(k)},a_{1}^{(k)},\dots ]$ and $p_{t}^{(k)}/q_{t}^{(k)} $ its $t$th convergent. Thus, $(mx+1-n)/(x-1)= p_{t_{k}}^{(k)}/q_{t_{k}}^{(k)}$ for some $t_{k}$. Therefore $q_{t_{k}}^{(k)}|x-1$, and so $x-1 \ge q_{t_{k}}^{(k)}$. On the other hand, with the help of  {\it Mathematica}, we get that
\[
 \min q_{230}^{(k)}>3.88 \cdot 10^{109}>2.27 \cdot 10^{105}>x-1,
\]
 therefore $1\le t_{k} \le 230$, for all $3\le k \le 242$. Using {\it Mathematica}, we observe that $a_{t_{k}}+1 \le \max \left\{ a_{t}^{(k)}\right\} <4.09 \cdot 10^{70} $, for $k \in \left\{3, \dots ,242\right\}$ and $t \in \left\{1, \dots ,231\right\}$. From the properties of continued fractions, we have 
\[
\left| \beta_{k}-\frac{mx+1-n}{x-1}\right|=\left|  \beta_{k}- \frac {p_{t_k}^{(k)}}{q_{t_k}^{(k)}}\right|>\frac{1}{4.09 \cdot 10^{70} (x-1)^{2}}
\]
which contradicts \eqref{eq38}. So, $x \le 150$ and $
k \in \left\{3,4,5\right\}$.

Now divide $(1+ \alpha^{-2x})$ in \eqref{upper}, we get
\begin{equation}\label{eq39}
\left| \alpha^{n-(mx+1) g^{1-x}(1+\alpha^{-2x})^{-1}} - 1\right|< \frac{2.11}{\alpha^{(m-2/2)}}.
\end{equation}
Next, we put $t:=mx+1-n$. Using \eqref{eq36}, we obtain
\[
g^{1-x} \alpha^{-t} -1 < 1.9 \cdot 10^{-10},
\]
yielding
\begin{align*}
t &> \frac{(x-1)\log g^{-1}}{\log \alpha} - \frac{\log (1+1.9 \cdot 10^{-10})}{\log \alpha}\\ 
&> 0.68x-0.69	
\end{align*}	  
and 
\[
g^{1-x} \alpha^{-t} -1 > -1.9 \cdot 10^{-10},
\]
yielding
\begin{align*}
t &< \frac{(x-1)\log g^{-1}}{\log \alpha} - \frac{\log (1-1.9 \cdot 10^{-10})}{\log \alpha}\\ 
&< 1.27x-1.26.
\end{align*}
Therefore, $t\in \left( \lfloor0.68x -0.69\rfloor ,\lfloor1.27x-1.26\rfloor \right)$. After performing a calculation for the range $ 20\le x \le 150, 3 \le k \le 5 $ and $0.68x-0.69 < t < 1.27x-1.26$, we get
\[
\min \left\{ \left| \frac{\alpha^{-t} g^{1-x}}{(1+ \alpha^{-x})}                      \right| -1 \right\} > 0.0003,
\]
which implies
\[
\frac{2.11}{\alpha^{(m-2)/2}}> 0.0003 \Rightarrow m \le 33 
\]
and that contradicts the fact $m \ge 1458$. Hence, the result is proved.


\vspace{10mm} \noindent \footnotesize
\begin{minipage}[b]{90cm}
\large{USTHB, Faculty of Mathematics, \\ 
LATN Laboratory, BP 32, El Alia, 16111, \\ 
Bab Ezzouar, Algiers, Algeria.\\
Email: hbensella@usthb.dz}
\end{minipage}

\vspace{05mm} \noindent \footnotesize
\begin{minipage}[b]{90cm}
\large{Department of Mathematics, \\ 
KIIT University, Bhubaneswar, \\ 
Bhubaneswar 751024, Odisha, India. \\
Email: bijan.patelfma@kiit.ac.in, iiit.bijan@gmail.com}
\end{minipage}

\vspace{05mm} \noindent \footnotesize
\begin{minipage}[b]{90cm}
\large{USTHB, Faculty of computer sciences,\\
BP 32, El Alia, 16111 Bab Ezzouar, Algiers, Algeria.\\
Email: dbehloul@yahoo.fr}
\end{minipage}
\end{document}